\documentclass[reqno,11pt]{amsart}
\usepackage{citesort}
    
\newtheorem{theorem}{Theorem}[section]
\newtheorem{proposition}[theorem]{Proposition}
\newtheorem{lemma}[theorem]{Lemma}

\theoremstyle{definition}
\newtheorem{definition}[theorem]{Definition}

\theoremstyle{remark}
\newtheorem{remark}[theorem]{Remark}

\numberwithin{equation}{section}

\def\R{{\mathbb{R}}}

\def\vect{\overrightarrow}
\def\H{{\vec{H}}}

\def\epsilon{{\varepsilon}}
\def\phi{{\varphi}}
\def\theta{{\vartheta}}

\DeclareMathOperator{\tr}{tr}

\def\ol#1{{\overline{#1}}}

\begin{document}
\title{Geometric invariants and principal configurations on spacelike surfaces immersed in ${\mathbb R}^{3,1}$}
\author{Pierre Bayard \& Federico S\'anchez-Bringas}
\address{Pierre Bayard, Instituto de
F\'{\i}sica y Matem\'aticas, U.M.S.N.H. Ciudad Universitaria, CP.
58040 Morelia, Michoac\'an, Mexico}
\email{bayard@ifm.umich.mx}
\address{Federico S\'anchez-Bringas, Facultad de Ciencias, UNAM, Mexico}
\email{sanchez@servidor.unam.mx}
\begin{abstract}
We first describe the numerical invariants attached to the second fundamental form of a spacelike surface in four-dimensional Min\-kowski space. We then study the configuration of the $\nu$-principal curvature lines on a  spacelike surface, when the normal field $\nu$ is lightlike (\textit{the lightcone configuration}). Some observations on the mean directionally curved lines and on the asymptotic lines on spacelike surfaces end the paper. 
\end{abstract}
\maketitle
\thispagestyle{empty}
\markboth{\textsl{P. Bayard, F. S\'anchez-Bringas}}
{\textsl{Invariants and principal configurations on spacelike surfaces in $\R^{3,1}$}}
\section*{Introduction} 
The Minkowski space $\R^{3,1}$ is the affine space $\R^4$ endowed with the metric $g=dx_1^2+dx_2^2+dx_3^2-dx_4^2.$ We say that a surface $M$ in $\R^{3,1}$ is spacelike if $g$ induces on $M$ a Riemannian metric. Thus, at each point $p$ of a spacelike surface $M$, the space $\R^{3,1}$ splits in
$$\R^{3,1}=T_pM\oplus N_pM,$$
where the tangent plane $T_pM$ and the normal plane $N_pM$ at $p$ are respectively equipped with a metric of signature $(2,0)$ and $(1,1).$ Denoting by $\ol D$ the usual covariant derivative on $\R^{3,1},$ the second fundamental form of $M$ at $p\in M$ is the bilinear map
$$II_p:T_pM\times T_pM\rightarrow N_pM,\ (X,Y)\mapsto (\ol D_X\ol Y)^{\perp}$$
which maps two tangent vectors $X,Y$ at $p$ to the normal component of the covariant derivative in the direction of $X$ of a local field $\ol Y$ of $M$ which extends $Y.$

In the first section of this article, we describe the numerical invariants of a quadratic map $\R^2\rightarrow\R^{1,1}$ modulo the natural action of the groups of isometries,  and classify the equivalence classes in terms of the invariants. As a byproduct of this systematic algebraic study, we obtain all the invariants of the second fundamental form of a spacelike surface: additionally to the lorentzian analogs of the four invariants defined on surfaces in four-dimensional euclidean space (see e.g. \cite{Little}), a new invariant appears in some exceptional cases, and completes the set of invariants (Theorem \ref{theorem algebraic classification}). In the second section we study the geometric properties of the curvature ellipse associated to a quadratic map: we give an intrinsic  equation of the curvature ellipse which naturally leads to the description of the ellipse in terms of the invariants (Theorem \ref{theorem ellipse classification}).

After these algebraic preliminaries, we study special lines on a spacelike surface. 
We first write down the normal form of the differential equation of the so-called lightcone configuration at an isolated umbilic point (Section \ref{normal form}), and then obtain from the Poincar\'e-Hopf Theorem a lower bound of the number of lightlike umbilic points on a generic closed spacelike surface (Theorems \ref{number lightlike umbilic points} and \ref{existence lightlike umbilic points}).
In the last section of the article we study the mean directionally curved lines and the asymptotic lines.

\section{Quadratic maps from $\R^2$ to $\R^{1,1}$}\label{section def invariants}
In this section, we study the vector space $Q(\R^2,\R^{1,1})$ of quadratic maps from $\R^2$ to $\R^{1,1}.$ We suppose that $\R^2$ and $\R^{1,1}$ are canonically oriented and time-oriented: a vector of $\R^{1,1}$ will be called future-directed if its second component in the canonical basis is positive. We consider the reduced (connected) groups of euclidean and lorentzian direct isometries of $\R^2$ and $\R^{1,1},$ $SO_2(\R)$ and $SO_{1,1}(\R).$ They act on $Q(\R^2,\R^{1,1})$ by composition
\begin{eqnarray*}
SO_{1,1}(\R)\times Q(\R^2,\R^{1,1})\times SO_2(\R)&\rightarrow &Q(\R^2,\R^{1,1})\\
(g_1,q,g_2)&\mapsto& g_1\circ q\circ g_2.
\end{eqnarray*}
We are interested in the description of the quotient set $$\displaystyle{SO_{1,1}\setminus Q(\R^2,\R^{1,1})/SO_2.}$$
We first introduce the algebraic invariants of a quadratic map and then give a description of the quotient set (Theorem \ref{theorem algebraic classification}). 

\subsection{Forms associated to a quadratic map}
We fix $q\in Q(\R^2,\R^{1,1}).$ If $\nu$ belongs to $\R^{1,1},$ we denote by $S_{\nu}$ the symmetric endomorphism of $\R^2$ associated to the real quadratic form $\langle q,\nu\rangle,$ and we define, for $\nu,\nu_1,\nu_2\in\R^{1,1},$
$$L_q(\nu):=\frac{1}{2}\tr (S_{\nu}),\ Q_q(\nu):=\det(S_{\nu})\mbox{ and } A_q(\nu_1,\nu_2):=\frac{1}{2}\left[S_{\nu_1},S_{\nu_2}\right].$$
Here $\left[S_{\nu_1},S_{\nu_2}\right]$ denotes the morphism $S_{\nu_1}\circ S_{\nu_2}-S_{\nu_2}\circ S_{\nu_1};$ it is skew-symmetric on $\R^2,$ and thus identifies with the real number $\alpha$ such that its matrix in the canonical basis of $\R^2$ is
$$\left[S_{\nu_1},S_{\nu_2}\right]=\left(\begin{array}{cc}0&-\alpha\\\alpha&0\end{array}\right).$$
In the sequel, we will implicitly make this identification. Thus, $L_q$ is a linear form, $Q_q$ is a quadratic form, and $A_q$ is a bilinear skew-symmetric form on $\R^{1,1}.$ These forms are linked together by the following lemma:
\begin{lemma}\label{Lemma identity Phi,A}
The quadratic form
$$\Phi_q:=L_q^2-Q_q$$
is non-negative and the following identity holds: for all $\nu_1,\nu_2\in\R^{1,1},$
\begin{equation}\label{identity Phi,A}
\Phi_q(\nu_1)\Phi_q(\nu_2)=\tilde\Phi_q(\nu_1,\nu_2)^2+A_q(\nu_1,\nu_2)^2,
\end{equation}
where $\tilde\Phi$ denotes the polar form of $\Phi.$
\end{lemma}
\begin{proof}
Let $F$ be the vector space of traceless symmetric operators on $\R^2.$ We first equip $F$ with a structure of oriented euclidean space. We define the scalar product $\langle.,.\rangle_F$ on $F$ as follows: if the operators are identified with their matrix in the canonical basis of $\R^2,$ we set, for all $S,T\in\ F,$ 
\begin{equation}\label{scalar product}
\langle S,T\rangle_F:=\frac{1}{2}\tr(ST).
\end{equation}
We choose the orientation on $F$ such that the orthonormal basis
$$E_1=\left(\begin{array}{cc} 1& 0\\0&-1\end{array}\right),\ E_2=\left(\begin{array}{cc} 0 & 1\\ 1&0\end{array}\right)$$
be positively oriented.  By a direct computation the mixed product $[.,.]_F$ on $F$ is given by: for all $S,T\in F,$ 
\begin{equation}\label{mixed product}
[S,T]_F=\frac{1}{2}[S,T].
\end{equation}
We recall that the mixed product $[.,.]_F$ is defined as the determinant in a positively oriented basis of $F,$ and that the bracket $[S,T]$ (without index) stands for the commutator $ST-TS.$

We first note the following interpretation of $\Phi:$ let $\nu\in\R^{1,1},$ and denote by $\lambda_1,\lambda_2$ the eigenvalues of $S_{\nu}.$ If $S_{\nu}^o$ denotes the traceless part of the symmetric operator $S_{\nu}$ then
$$S_{\nu}^o=S_{\nu}-\frac{1}{2}\tr S_{\nu}.\mbox{Id},$$
and the eigenvalues of $S_{\nu}^o$ are $\displaystyle{\pm\frac{\lambda_1-\lambda_2}{2}.}$ Thus, in view of (\ref{scalar product}),
$$\|S_{\nu}^o\|_F^2=\frac{1}{4}(\lambda_1-\lambda_2)^2.$$
By the very definition of $\Phi$ this last quantity is $\Phi(\nu),$ and we thus get
\begin{equation}\label{interpretation Phi}
\Phi(\nu)=\|S_{\nu}^o\|_F^2\mbox{ and }\tilde\Phi(\nu_1,\nu_2)=\langle S_{\nu_1}^o,S_{\nu_2}^o\rangle_F.
\end{equation}
We now give an interpretation of the form $A:$  observing that 
$$\left[S_{\nu_1},S_{\nu_2}\right]=\left[S_{\nu_1}^o,S_{\nu_2}^o\right],$$
 we get from the definition of $A$ and (\ref{mixed product}) that
\begin{equation}\label{interpretation A}
A(\nu_1,\nu_2)=\left[S_{\nu_1}^o,S_{\nu_2}^o\right]_F.
\end{equation}
Identity (\ref{identity Phi,A}) is then straightforward: the Lagrange identity in $F$ reads
\begin{equation}\label{Lagrange identity}
\|S\|_F^2\|T\|_F^2=\langle S,T\rangle_F^2+[S,T]_F^2.
\end{equation}
Applying (\ref{Lagrange identity}) to $S_{\nu_1}^o$ and $S_{\nu_2}^o$ and using (\ref{interpretation Phi}) and (\ref{interpretation A}) we get (\ref{identity Phi,A}).
\end{proof}
The forms $L_q,\ \Phi_q,\ A_q$ are invariant by the $SO_2$-action on $q:$ for all $g\in SO_2(\R),$
$$L_{q\circ g}=L_q,\ \Phi_{q\circ g}=\Phi_q\mbox{ and }A_{q\circ g}=A_q.$$
 Moreover, these forms contain all the information on $q$ modulo this action: setting
 $$P=\{(L,\Phi,A)\in {\R^{1,1}}^*\times S^2{\R^{1,1}}^*\times \Lambda^2{\R^{1,1}}^*:\ \Phi\mbox{ is non-negative, and }(\ref{identity Phi,A})\mbox{ holds}\},$$
where $S^2{\R^{1,1}}^*$ and $\Lambda^2{\R^{1,1}}^*$ stand for the sets of symmetric and skew-symmetric bilinear forms on $\R^{1,1},$ the following result holds:
\begin{lemma}
The map 
\begin{eqnarray*}
\Theta:Q(\R^2,\R^{1,1})/SO_2&\rightarrow& P\\
{[q]} &\mapsto & \left(L_{[q]},\ \Phi_{[q]},\ A_{[q]}\right)
\end{eqnarray*}
 is bijective.
\end{lemma}
\begin{proof}
We first give a formula which permits to recover $[q]$ from its associated forms $L,\Phi,A$ (formula (\ref{S function Phi, A}) below). We keep the notation introduced in the proof of Lemma \ref{Lemma identity Phi,A} above.

If $\Phi$ is not the null form, we fix $\nu_o\in\R^{1,1}$ such that $\Phi(\nu_o)=1. $ The eigenvalues of $S_{\nu_o}$ are the roots of the polynomial 
$$X^2-2L(\nu_o)X+Q(\nu_o)$$
and are thus $L(\nu_o)\pm1.$ We choose a positively oriented orthonormal basis $(e_1,e_2)$ of $\R^2$ such that  in this basis 
$$S_{\nu_o}=L(\nu_o)I+E_1.$$ 
For all $\nu\in \R^{1,1},$ the matrix of $S_{\nu}$ may be a priori written in the form
$$S_{\nu}=L(\nu)I+a_{\nu}E_1+b_{\nu}E_2,$$
where $a_{\nu}$ and $b_{\nu}$ belong to $\R.$  We necessarily have
\begin{equation}
a_{\nu}=\langle S_{\nu_o}^o,S_{\nu}^o\rangle_F=\tilde\Phi(\nu_o,\nu),
\end{equation}
and
\begin{equation}
b_{\nu}=\left[S_{\nu_o}^o,S_{\nu}^o\right]_F=A(\nu_o,\nu).
\end{equation}
Thus, in $(e_1,e_2),$
\begin{equation}\label{S function Phi, A}
S_{\nu}=L(\nu)I+\tilde\Phi(\nu_o,\nu)E_1+A(\nu_o,\nu)E_2.
\end{equation}
If now $\Phi$ is the null form, then, for all $\nu\in\R^{1,1},$ $S_{\nu}$ is an homothetie. Thus $S_{\nu}=L(\nu)I,$ and (\ref{S function Phi, A}) also holds in that case (observe that $A$ is null by (\ref{identity Phi,A})). Note that here $\nu_o$ and $(e_1,e_2)$ may be arbitrarily chosen. 

Formula (\ref{S function Phi, A}) proves that the map $\Theta$ is injective

We now prove that $\Theta$ is surjective. We fix $(L,\Phi,A)\in P.$ We first assume that $\Phi$ is not the null form, and we fix $\nu_o\in\R^{1,1}$ such that $\Phi(\nu_o)=1.$  We denote by $(e_1,e_2)$ the canonical basis of $\R^2,$ and, for all $\nu\in\R^{1,1},$ we define $S_{\nu}$  the symmetric operator on $\R^2$ whose matrix in $(e_1,e_2)$ is given by (\ref{S function Phi, A}). The map $\nu\mapsto S_{\nu}$ is linear and, for all $\nu\in\R^{1,1},$ 
$$\frac{1}{2}\tr(S_{\nu})=L(\nu),$$ 
\begin{eqnarray*}
\left(\frac{1}{2}\tr(S_{\nu})\right)^2-\det(S_{\nu})&=&L^2(\nu)-\left|\begin{array}{cc}
L(\nu)+\tilde\Phi(\nu_o,\nu) & A(\nu_o,\nu)\\
A(\nu_o,\nu) & L(\nu)-\tilde\Phi(\nu_o,\nu) 
\end{array}\right|\\
&=&\Phi(\nu),
\end{eqnarray*}
and
$$\frac{1}{2}\left[S_{\nu_o},S_{\nu}\right]= \left[S_{\nu_o}^o,S_{\nu}^o\right]_F=\left|\begin{array}{cc} 1& \tilde\Phi(\nu_o,\nu)\\0&A(\nu_o,\nu)\end{array}\right|=A(\nu_o,\nu).$$
Note that the last identity implies that
$$\frac{1}{2}\left[S_{\nu_1},S_{\nu_2}\right]=A(\nu_1,\nu_2)$$
for all $\nu_1,\nu_2\in\R^{1,1},$ since the two terms of this identity are skew-symmetric in $\nu_1,\nu_2.$ Thus, the quadratic map $q:\R^2\rightarrow\R^{1,1}$ such that $\langle q,\nu\rangle=\langle S_{\nu}.,.\rangle$ satisfies $\Theta([q])=(L,\Phi,A).$
\end{proof}
By the natural $SO_{1,1}$-action on $Q(\R^2,\R^{1,1})/SO_2,$ the forms $L_{[q]}, \Phi_{[q]}$ transform as
$$L_{g.[q]}= L_{[q]}\circ g^{-1}, \Phi_{g.[q]}=\Phi_{[q]}\circ g^{-1},$$
whereas the form $A_{[q]}$ is kept invariant. Let us consider the following action of $SO_{1,1}$ on $P:$ for all $g\in SO_{1,1},$ for all $(L,\Phi,A)\in P,$
\begin{equation}\label{action SO1,1}
g.(L,\Phi,A):=(L\circ g^{-1},\Phi\circ g^{-1},A).
\end{equation}
The map $\Theta$ is  $SO_{1,1}$-equivariant and thus induces a bijective map
\begin{equation}\label{bijection theta bar}
\overline{\Theta}:SO_{1,1}\setminus Q(\R^2,\R^{1,1})/SO_2\rightarrow SO_{1,1}\setminus P.
\end{equation}
Observe that formula (\ref{identity Phi,A}) permits to recover $A$ (up to sign) from $\Phi.$ The description of the quotient set $\displaystyle{SO_{1,1}\setminus Q(\R^2,\R^{1,1})/SO_2}$ will thus be achieved with the simultaneous reduction of  the forms $L_{[q]}$ and $\Phi_{[q]}$ in $\R^{1,1}.$

\subsection{Invariants on the quotient set}
For the discussion, we first define invariants on  $\displaystyle{SO_{1,1}\setminus Q(\R^2,\R^{1,1})/SO_2}$ associated to $L_{[q]},\ Q_{[q]},\ A_{[q]}$ and $\Phi_{[q]}:$

\begin{definition}
We consider:
\\
\\1- the vector $\H\in\R^{1,1}$ such that, for all $\nu\in \R^{1,1}, $ $\langle\H,\nu\rangle=L_{[q]}(\nu),$
and its norm $$|\H|^2:=\langle\H,\H\rangle;$$
2- the two real numbers 
$$K:=\tr Q_{[q]}\mbox{ and } \Delta:=\det Q_{[q]},$$
where $\tr Q_{[q]}\mbox{ and } \det Q_{[q]}$ are the trace and the determinant of the symmetric operator of $\R^{1,1}$ associated to $Q_{[q]}$ by the scalar product on $\R^{1,1};$
\\
\\3- the real number $K_N$ such that
$$A_{[q]}=\frac{1}{2}K_N\omega_0,$$
where $\omega_0$ is the determinant in the canonical basis of $\R^{1,1}$ (the canonical volume form on $\R^{1,1}).$
\end{definition}
The numbers $|\H|^2,K,K_N,\Delta$ are invariant by the $SO_{1.1}$-action on $[q]\in Q(\R^2,\R^{1,1})/SO_2,$ and thus define invariants on $\displaystyle{SO_{1,1}\setminus Q(\R^2,\R^{1,1})/SO_2.}$

The invariants of the quadratic form $\Phi_{[q]}$ have the following simple expressions in terms of these invariants:
\begin{lemma} \label{Lemma tr det Phi}
Let $\Phi_{[q]}:=L^2_{[q]}-Q_{[q]}$ and $u_{\Phi}$ the symmetric operator on $\R^{1,1}$ such that  $$\Phi_{[q]}(\nu)=\langle u_{\Phi}(\nu),\nu\rangle$$ for all $\nu\in\R^{1,1}.$ Denoting by  $\tr\Phi_{[q]}$ and $\det\Phi_{[q]}$ its trace and its determinant, we get:
\begin{equation}\label{lemma tr det Phi}
\tr\Phi_{[q]}=|\H|^2-K\mbox{ and }\det\Phi_{[q]}=-\frac{1}{4}K_N^2.
\end{equation}
\end{lemma}
Here and below, we denote by $(u_1,u_2)$ the canonical basis of $\R^{1,1},$ and by $(N_1,N_2)$ the basis defined by
$$N_1=\frac{\sqrt{2}}{2}\left(u_1+u_2\right),\ N_2=\frac{\sqrt{2}}{2}\left(u_2-u_1\right).$$ 
The basis $(N_1,N_2)$ is positively oriented, and $N_1$ and  $N_2$ are null and future-directed vectors  such that $\langle N_1,N_2\rangle=-1.$ Moreover, to simplify the notation the forms $L_{[q]},Q_{[q]},\Phi_{[q]}$ will be denoted by $L,Q,\Phi.$
\begin{proof}
The matrix of $u_{\Phi}$ in the basis $(N_1,N_2)$ is
\begin{equation}
-\left(\begin{array}{cc}
\tilde\Phi(N_1,N_2)&\Phi(N_2)\\
\Phi(N_1)&\tilde\Phi(N_1,N_2)
\end{array}
\right).
\end{equation}
Thus, denoting by $\tilde Q$ the polar form of $Q,$  
\begin{eqnarray*}
\tr\Phi=-2\tilde\Phi(N_1,N_2)&=&-2L(N_1)L(N_2)+2\tilde Q(N_1,N_2)\\
&=&-2L(N_1)L(N_2)-\tr Q\\
&=&|\H|^2-K,
\end{eqnarray*}
and
$$\det\Phi=\tilde\Phi(N_1,N_2)^2-\Phi(N_1)\Phi(N_2)=-\frac{1}{4}K_N^2$$
by (\ref{identity Phi,A}).
\end{proof}
We now give some formulas to compute these invariants.  We write the quadratic map $q$ in the form $q=q_1N_1+q_2N_2,$ and we denote by  
$$\left(\begin{array}{cc}
x&z\\
z&y
\end{array}\right),\ \left(\begin{array}{cc}
u&w\\
w&v
\end{array}\right)$$
the matrices of $q_1$ and $q_2$ respectively in the canonical basis $(e_1,e_2)$ of $\R^2.$ 
\begin{lemma} \label{invariants function xyzuvw} 
The following formulas hold:
\begin{equation}\label{expr H}
\H=\frac{1}{2}(x+y)N_1+\frac{1}{2}(u+v)N_2,
\end{equation}
\begin{equation}\label{expr H2 K K_N}
|\H|^2=-\frac{1}{2}(x+y)(u+v) ,\ K=-xv-uy+2zw,\ K_N=-z(u-v)+w(x-y),
\end{equation}
and
$$\Delta=\frac{1}{4}K^2-(uv-w^2)(xy-z^2).$$ 
\end{lemma}
\begin{proof}
Let $\nu=\nu_1N_1+\nu_2N_2\in\R^{1,1}.$ The matrix of $q_{\nu}$ in $(e_1,e_2)$ is
\begin{equation}\label{matrix q_nu}
-\left(\begin{array}{cc}
x\nu_2+u\nu_1&z\nu_2+w\nu_1\\
z\nu_2+w\nu_1& y\nu_2+v\nu_1
\end{array}\right).
\end{equation}
Its trace is
\begin{equation}\label{formula L}
2L(\nu)=-((u+v)\nu_1+(x+y)\nu_2),
\end{equation}
which gives (\ref{expr H}) and the expression of $|\H|^2.$ Its determinant is
\begin{equation}\label{formula Q}
Q(\nu)=(uv-w^2)\nu_1^2+(xy-z^2)\nu_2^2+(xv+uy-2zw)\nu_1\nu_2.
\end{equation}
The matrix in $(N_1,N_2)$ of the symmetric operator $u_Q$ associated to $Q$ is
\begin{equation}\label{matrix u_Q}
-\left(\begin{array}{cc}
\frac{1}{2}(xv+uy)-zw&xy-z^2\\
uv-w^2&\frac{1}{2}(xv+uy)-zw
\end{array}\right).
\end{equation}
Taking the trace and the determinant we obtain the expressions of $K$ and $\Delta.$ From (\ref{matrix q_nu}), 
$$S_{N_1}=-\left(\begin{array}{cc}u&w\\w&v\end{array}\right)\mbox{ and } S_{N_2}=-\left(\begin{array}{cc}x&z\\z&y\end{array}\right).$$
Thus 
$$[S_{N_1},S_{N_2}]=\left(\begin{array}{cc}0&z(u-v)-w(x-y)\\
-z(u-v)+w(x-y)&0\end{array}\right),$$
and we obtain the expression of $K_N.$
\end{proof}
We now introduce some notation which is especially adapted to the study of the form $\Phi.$ This notation is linked to the description of the curvature ellipse; see Section \ref{section curvature ellipse}. We set
\begin{equation}\label{def h1 h2}
h_1=\frac{1}{2}(x+y),\ h_2=\frac{1}{2}(u+v),
\end{equation}
and
\begin{equation}\label{def mu nu}
\mu=\frac{1}{2}(x-y),\ \nu=z,\ \mu'=\frac{1}{2}(u-v),\ \nu'=w.
\end{equation}
\begin{lemma} \label{invariants function mu nu}
The matrix of $\Phi$ in $(N_1,N_2)$ is
\begin{equation}\label{matrix Phi}
\left(\begin{array}{cc} 
\nu'^2+\mu'^2&\nu'\nu+\mu\mu'\\
\nu'\nu+\mu\mu'&\nu^2+\mu^2
\end{array}\right),
\end{equation}
and the following formulas hold:
\begin{equation}\label{formula 1 mu nu}
\nu'\nu+\mu\mu'=\frac{1}{2}(K-|\H|^2),\ (\nu\mu'-\nu'\mu)^2=\frac{1}{4}K_N^2,
\end{equation}
and
\begin{equation}\label{formula 2 mu nu}
(\nu^2+\mu^2)(\nu'^2+\mu'^2)=\frac{1}{4}\left((K-|\H|^2)^2+K_N^2\right).
\end{equation}
\end{lemma}
\begin{proof}
Recalling the formulas (\ref{formula L}) and (\ref{formula Q}), we obtain
\begin{eqnarray*}
\Phi(\nu)&=&\left(\frac{1}{4}(u+v)^2-(uv-w^2)\right)\nu_1^2+\left(\frac{1}{4}(x+y)^2-(xy-z^2)\right)\nu_2^2\\&&+\left(\frac{1}{2}(x+y)(u+v)-(xv+uy-2zw)\right)\nu_1\nu_2.
\end{eqnarray*}
Using (\ref{def mu nu}), we easily get
$$\Phi(\nu)=(\nu'^2+\mu'^2)\nu_1^2+(\nu^2+\mu^2)\nu_2^2+2(\nu'\nu+\mu\mu')\nu_1\nu_2,$$
and the matrix (\ref{matrix Phi}) of $\Phi$ in $(N_1,N_2).$ The matrix of $u_{\Phi}$ in the basis $(N_1,N_2)$ is then given by
\begin{equation}\label{matrix u_Phi}
-\left(\begin{array}{cc} 
\nu'\nu+\mu\mu' & \nu^2+\mu^2\\
\nu'^2+\mu'^2&\nu'\nu+\mu\mu'
\end{array}\right).
\end{equation}
Lemma \ref{Lemma tr det Phi} and the Lagrange identity imply the claimed formulas.
\end{proof}

\subsection{The simultaneous reduction of $\Phi_{[q]}$ and $L_{[q]}$}
We now reduce the operator $u_{\Phi}.$ Here, in contrast with the euclidean case, it may be not diagonalizable. 
\begin{lemma}\label{case Phi diag}
The operator $u_{\Phi}$ is diagonalizable if and only if $\Phi=0$ or
\begin{equation}\label{condition u_Phi diagonalisable}
K_N^2+\left(|\H|^2-K\right)^2>0.
\end{equation}
If (\ref{condition u_Phi diagonalisable}) holds, there is a unique basis of unit eigenvectors of $u_{\Phi},$ $(\tilde u_1,\tilde u_2),$ which is positively oriented and such that $\tilde u_1$ is spacelike  and $\tilde u_2$ is timelike and future-directed; moreover, the basis $(\tilde u_1,\tilde u_2)$ is orthogonal and the matrix of $u_{\Phi}$ in $(\tilde u_1,\tilde u_2)$ is diagonal with diagonal entries $(a^2,-b^2),$ where
\begin{equation}\label{equation a}
a^2=\frac{1}{2}\left(\left(|\H|^2-K\right)+\sqrt{K_N^2+\left(|\H|^2-K\right)^2}\right),
\end{equation}
\begin{equation}\label{equation b}
b^2=\frac{1}{2}\left(-\left(|\H|^2-K\right)+\sqrt{K_N^2+\left(|\H|^2-K\right)^2}\right).
\end{equation}
In $(\tilde u_1,\tilde u_2),$ the vector $\H$ is $\alpha \tilde u_1+\beta \tilde u_2,$ with
\begin{equation}\label{equation alpha}
\alpha^2=\frac{1}{\sqrt{K_N^2+\left(|\H|^2-K\right)^2}}\left(\Delta+a^2|\H|^2+\frac{1}{4}K_N^2\right),
\end{equation}
\begin{equation}\label{equation beta}
\beta^2=\frac{1}{\sqrt{K_N^2+\left(|\H|^2-K\right)^2}}\left(\Delta-b^2|\H|^2+\frac{1}{4}K_N^2\right).
\end{equation}
\end{lemma}
\begin{proof}
By Lemma \ref{Lemma tr det Phi}, the characteristic polynomial of $u_\Phi$ is
\begin{equation}\label{pol car}
\det(u_{\Phi}-X .\mbox{Id})=X^2-(|\H|^2-K)X-\frac{1}{4}K_N^2.
\end{equation}
Condition  (\ref{condition u_Phi diagonalisable}) means that its discriminant is positive.
Thus, if  (\ref{condition u_Phi diagonalisable}) holds, the eigenvalues of $u_{\Phi}$ are distinct, $u_{\Phi}$ is diagonalizable, and its eigenvalues $a^2,\ -b^2$ are given by (\ref{equation a}), (\ref{equation b}).
Conversely, we assume that $u_{\Phi}$ is diagonalizable, and we denote by $\lambda_1,\lambda_2$ its eigenvalues. Since $\lambda_1\lambda_2=\det u_{\Phi}\leq 0,$ we see that $\lambda_1\neq\lambda_2$ or $u_{\Phi}=0.$ Thus, if $u_{\Phi}\neq 0,$ the characteristic polynomial (\ref{pol car}) has a positive discriminant, which gives condition (\ref{condition u_Phi diagonalisable}).  

We now explicit the eigenvectors of $u_{\Phi}.$ Recall that the matrix of $u_{\Phi}$ in the basis $(N_1,N_2)$ is given by (\ref{matrix u_Phi}). Setting
\begin{eqnarray*}
\xi^{(1)}&=&\sqrt{\mu^2+\nu^2}N_1-\sqrt{\mu'^2+\nu'^2}N_2\\
\xi^{(2)}&=&\sqrt{\mu^2+\nu^2}N_1+\sqrt{\mu'^2+\nu'^2}N_2,
\end{eqnarray*}
we verify by a direct computation that 
\begin{eqnarray*}
u_{\Phi}(\xi^{(1)})&=&\left(-(\nu'\nu+\mu\mu')+\sqrt{\nu^2+\mu^2}\sqrt{\mu'^2+\nu'^2}\right)\xi^{(1)}\\
&=&a^2\xi^{(1)},
\end{eqnarray*}
and 
\begin{eqnarray*}
u_{\Phi}(\xi^{(2)})&=&-\left((\nu'\nu+\mu\mu')+\sqrt{\nu^2+\mu^2}\sqrt{\mu'^2+\nu'^2}\right)\xi^{(2)}\\
&=&-b^2\xi^{(2)}.
\end{eqnarray*}
Thus, setting $$\tilde u_1=\frac{\xi^{(1)}}{|\xi^{(1)}|}\mbox{ and } \tilde u_2=\frac{\xi^{(2)}}{|\xi^{(2)}|},$$
we obtain the basis $(\tilde u_1,\tilde u_2)$ as claimed in the lemma.

We now prove identities (\ref{equation alpha}) and (\ref{equation beta}): we consider
$$\xi=\xi_1N_1+\xi_2 N_2,$$
with $\xi_1=\sqrt{\mu^2+\nu^2},\ \xi_2=\pm \sqrt{\mu'^2+\nu'^2}.$ From Lemma \ref{invariants function xyzuvw},
$$\Delta=\frac{1}{4}K^2-(xy-z^2)(uv-w^2),$$
with $xy-z^2=h_1^2-\xi_1^2,$ $uv-w^2=h_2^2-\xi_2^2$ (recall (\ref{def h1 h2}) and (\ref{def mu nu})).
Since, by a direct computation,
\begin{eqnarray}
(h_1^2-\xi_1^2)(h_2^2-\xi_2^2)&=&(h_1h_2-\xi_1\xi_2)^2-(h_1\xi_2-h_2\xi_1)^2\\
&=&\frac{1}{4}(|\H|^2-|\xi|^2)^2-[\H,\xi]^2,
\end{eqnarray}
we have
\begin{eqnarray}
\Delta&=&\frac{1}{4}K^2-\frac{1}{4}(|\H|^2-|\xi|^2)^2+[\H,\xi]^2\\
&=&\frac{1}{4}(K-|\H|^2+|\xi|^2)(K+|\H|^2-|\xi|^2)+[\H,\xi]^2.\label{Delta function xi}
\end{eqnarray}
We first take $\xi=\sqrt{\mu^2+\nu^2}N_1+\sqrt{\mu'^2+\nu'^2}N_2.$ We have 
$$-|\xi|^2=\sqrt{K_N^2+\left(|\H|^2-K\right)^2}=a^2+b^2.$$
Observing that
$$-K+|\H|^2=a^2-b^2\mbox{ and }K_N^2=4a^2b^2,$$
we obtain 
\begin{equation}\label{formula Delta a}
\Delta=-a^2|\H|^2-\frac{1}{4}K_N^2+[\H,\xi]^2.
\end{equation}
Since $\xi$ is proportional to $\tilde u_2,$ we have $[\H,\xi]^2=-\alpha^2|\xi|^2,$ and we finally obtain (\ref{equation alpha}).

If we now take $\xi=\sqrt{\mu^2+\nu^2}N_1-\sqrt{\mu'^2+\nu'^2}N_2,$ we have $|\xi|^2=a^2+b^2,$ and analogous computations give
\begin{equation}\label{formula Delta b}
\Delta=b^2|\H|^2-\frac{1}{4}K_N^2+[\H,\xi]^2.
\end{equation}
Here $\xi$ is proportional to $\tilde u_1,$ and we get (\ref{equation beta}).
\end{proof}

\begin{lemma}\label{case Phi no diag}
The operator $u_{\Phi}$ is not diagonalizable if and only if 
\begin{equation}\label{condition Phi no diag}
\Phi\neq 0 \mbox{ and }K_N=|\H|^2-K=0.
\end{equation}
 In that case there is a unique positively oriented basis $(\tilde N_1,\tilde N_2)$ of $\R^{1,1},$ with $\tilde N_1$ and $\tilde N_2$ null and future-directed vectors satisfying $\langle \tilde N_1,\tilde N_2\rangle=-1,$ such that the matrix of $u_{\Phi}$ in $(\tilde N_1,\tilde N_2)$ is
$$A_1=\left(\begin{array}{c c} 0&0\\ -1&0\end{array}\right)\mbox{ or }A_2=\left(\begin{array}{c c} 0&-1\\ 0&0\end{array}\right).$$
Suppose that the matrix of $u_{\Phi}$ is $A_1.$ Then the vector $\H$ reads 
\begin{equation}\label{equation h1}
\H=\pm\left(\sqrt{\Delta}\tilde N_1-\frac{K}{2\sqrt{\Delta}}\tilde N_2\right)
\end{equation}
 if $\Delta\neq 0,$ and is of the form
\begin{equation}\label{definition zeta}
\H=\zeta\tilde N_2
\end{equation}
 if $\Delta=0,$ where the number $\zeta\in\R$ defines a new invariant (the four invariants $|\H|^2,$ $K,$ $K_N,$ $\Delta$ vanish in that case). 

If the matrix of $u_{\Phi}$ is $A_2,$ the same holds, switching the indices 1 and 2.
\end{lemma}

\begin{proof}
The condition (\ref{condition Phi no diag}) directly follows from the condition (\ref{condition u_Phi diagonalisable}) in Lemma \ref{case Phi diag}. Recall the expression (\ref{matrix u_Phi}) of the matrix of $u_{\Phi}$ in the basis $(N_1,N_2).$ From Lemma \ref{invariants function mu nu} the condition (\ref{condition Phi no diag}) implies that
$$(\nu^2+\mu^2)(\nu'^2+\mu'^2)=0,\ \nu\nu'+\mu\mu'=0,\ \nu\mu'-\nu'\mu=0.$$
We suppose that $\nu^2+\mu^2=0.$ The matrix of $u_{\Phi}$ in $(N_1,N_2)$ reads
$$-\left(\begin{array}{cc}0& 0\\\nu'^2+\mu'^2& 0\end{array}\right).$$
Since $\Phi\neq 0,$ we have $\nu'^2+\mu'^2\neq 0.$ We consider the basis
$$\tilde N_1:=\frac{1}{\sqrt{\nu'^2+\mu'^2}}N_1,\ \tilde N_2:=\sqrt{\nu'^2+\mu'^2}N_2.$$
In $(\tilde N_1,\tilde N_2)$ the matrix of $u_{\Phi}$ is the matrix $A_1.$ Note that such a basis is clearly unique.

We now prove the second part of the lemma. We denote by $(\tilde h_1,\tilde h_2)$ the coordinates of $\H$ in $(\tilde N_1,\tilde N_2).$ Using (\ref{Delta function xi}), with here $\xi=\tilde N_2,$ and since $K-|\H|^2=0,$ $|\xi|^2=0$ and $[\H,\xi]^2=\tilde h_1^2,$ we obtain $\tilde h_1^2=\Delta.$ Since moreover $|\H|^2=-2\tilde h_1\tilde h_2=K,$ we obtain (\ref{equation h1}).

If $\Delta=0,$ then $\tilde h_1=0$ and all the previously defined invariants vanish. $\zeta=\tilde h_2$ defines a new invariant.
\end{proof}
\subsection{The classification}
We now gather the results obtained in Lemma \ref{case Phi diag} and Lemma \ref{case Phi no diag}. To state the main theorem we introduce the following notation: we denote by $G$ the subgroup of transformations of $\R^{1,1}$ generated by the reflections with respect to the lines $\R.u_1$ and $\R.u_2;$ identifying these transformations with their matrix in $(u_1,u_2)$, we have
$$G=\left\{\left(\begin{array}{cc}1&0\\0&1\end{array}\right),\left(\begin{array}{cc}-1&0\\0&-1\end{array}\right),\left(\begin{array}{cc}1&0\\0&-1\end{array}\right),\left(\begin{array}{cc}-1&0\\0&1\end{array}\right)\right\}.$$
This group acts on the quotient set $\displaystyle{SO_{1,1}\setminus Q(\R^2,\R^{1,1})/SO_2}$ by composition since $gSO_{1,1}g^{-1}=SO_{1,1}$ for all $g\in G.$ We denote by $G'$ the subgroup of $G$ generated by the reflection with respect to $\R.u_2:$
$$G'=\left\{\left(\begin{array}{cc}1&0\\0&1\end{array}\right),\left(\begin{array}{cc}-1&0\\0&1\end{array}\right)\right\}.$$
\begin{theorem}\label{theorem algebraic classification}
Let $[q]\in \displaystyle{SO_{1,1}\setminus Q(\R^2,\R^{1,1})/SO_2},$ and $K,K_N,|\H|^2,\Delta$ its invariants.
\\
\\1- If $K_N\neq0$ or $|\H|^2-K\neq 0,$ the invariants $K,K_N,|\H|^2,\Delta$ determine $[q]$ up to the action of the group $G.$
\\
\\2- If $K_N=0,\ |\H|^2-K=0,$ and $\Phi\neq 0,$ then
\\a- if $\Delta\neq 0,$  the invariants $\Delta,K$ determine $[q]$ up to the action of $G.$
\\b- if $\Delta=0,$ then the four invariants vanish, and the new invariant $\zeta$ defined in Lemma \ref{case Phi no diag} determines $[q]$ up to the action of $G'.$
\\
\\3- If $\Phi=0$ and $[q]\neq 0,$ the invariant $|\H|^2$ determines $[q]$ up to the action of $G.$ 
\end{theorem}
\begin{remark}
In fact, if $K_N\neq 0,$ the invariants $K,K_N,|\H|^2$ and $\Delta$ determine $[q]$ up to the action of $\pm id_{\R^{1,1}}$ only, since the action of a reflection of $\R^{1,1}$ changes $K_N$ to $-K_N.$
\end{remark}
\begin{proof}
First recall the definition of the bijective map $\overline\Theta$ in (\ref{bijection theta bar}). 

1- If $K_N\neq 0$ or $|\H|^2-K\neq 0,$ by Lemma \ref{case Phi diag}, $\overline{\Theta}([q])$ is the class of $(L,\Phi,A)\in P$ where the forms $L,\Phi$ and $A$ are defined in the \emph{canonical basis} $(u_1,u_2)$ of $\R^{1,1}$ by
$$L=(\alpha,-\beta),\ \Phi=\left(\begin{array}{cc}a^2&0\\0&b^2\end{array}\right)\mbox{ and }A=\frac{1}{2}K_N\left(\begin{array}{cc}0&1\\-1&0\end{array}\right),$$
with $a^2,b^2,\alpha,\beta$ satisfying (\ref{equation a}), (\ref{equation b}), (\ref{equation alpha}) and (\ref{equation beta}) (more precisely, recalling (\ref{action SO1,1}), if $g\in SO_{1,1}$ is such that $g(u_1)=\tilde u_1,$ $g(u_2)=\tilde u_2,$ where $(\tilde u_1,\tilde u_2)$ is the basis given by Lemma \ref{case Phi diag}, we have $g.(L,\Phi,A)=(L_q,\Phi_q,A_q))$. Since $\alpha$ and $\beta$ are determined up to sign by (\ref{equation alpha}) and (\ref{equation beta}), four classes correspond to the given set of invariants (note that classes coincide in the particular case where $\alpha$ or $\beta$ vanishes). They are obtained from one of them by the action of $G.$
\\2- If $K_N=0,$ $|\H|^2-K=0$ and $\Phi\neq 0,$ then, by Lemma \ref{case Phi no diag}, $\overline{\Theta}([q])$ is the class of $(L,\Phi,A)\in P$ where, in the basis $(N_1,N_2)$ of $\R^{1,1},$
$$L=-(h_2,h_1),\ \Phi=\left(\begin{array}{cc}1&0\\0&0\end{array}\right)\mbox{ or }\left(\begin{array}{cc}0&0\\0&1\end{array}\right)\mbox{ and }A=0.$$
If $\Phi$ is the first matrix, $(h_1,h_2)$ is given by (\ref{equation h1}) (if $\Delta\neq 0$) or by (\ref{definition zeta}) (if $\Delta=0$), and if $\Phi$ is the second matrix, we have to switch the indices 1 and 2 in (\ref{equation h1}), (\ref{definition zeta}).  Thus, if $\Delta\neq 0$ four classes correspond to the given set of invariants $\Delta,K$ (two classes, given by the choice of the $\pm$ sign in (\ref{equation h1}), for each one of the two possible matrices for $\Phi$) and if $\Delta=0$ two classes correspond to the invariant $\zeta$ (by (\ref{definition zeta}), only one class for each one of the two possible matrices for $\Phi$). They differ by the action of $G$ and $G'$ respectively, and we get the result. 
\\3- If $\Phi=0$ and $[q]\neq 0,$ $\overline{\Theta}([q])$ is the class of $(L,0,0)\in P$ with $L=\langle\H,.\rangle,$ where 
$$\H=\pm N_1,\pm N_2, \pm\lambda u_1,\pm\lambda u_2,$$
with $\lambda>0.$   
\end{proof}
\section{The curvature ellipse}\label{section curvature ellipse}
In this  section we describe the geometric properties of the curvature ellipse associated to a quadratic map in terms of its invariants (Theorem \ref{theorem ellipse classification}). The curvature ellipse associated to the second fundamental form of a spacelike surface in four-dimensional Minkowski space was first considered in \cite{IPR}. The curvature ellipse $E$ associated to $q:\R^2\rightarrow\R^{1,1}$ is defined as the subset of $\R^{1,1}$
$$E:=\left\{q(v):\ v\in\R^2,\ |v|=1\right\}.$$
It is parameterized by 
$$\theta\in\R\mapsto \H+X(\theta)N_1+Y(\theta)N_2\in\R^{1,1}$$ 
with $\H=h_1N_1+h_2N_2$ and
\begin{equation}\label{param ellipse theta}
X(\theta)=\mu\cos 2\theta+\nu\sin 2\theta,\ Y(\theta)=\mu'\cos 2\theta+\nu'\sin 2\theta,
\end{equation}
where the coefficients $h_1,h_2,\mu,\nu,\mu'$ and $\nu'$ are defined by the expressions (\ref{def h1 h2}) and (\ref{def mu nu}).

The two next lemmas give an explicit equation of the curvature ellipse:
\begin{lemma}
If $\H+XN_1+YN_2$ belongs to the curvature ellipse, then $(X,Y)$ verifies
\begin{eqnarray}\label{equation ellipse mu nu}
\varepsilon(X,Y)&=&\left(\nu'^2+\mu'^2\right)X^2+\left(\nu^2+\mu^2\right)Y^2-2\left(\nu'\nu+\mu\mu'\right)XY\nonumber\\&=&\left(\nu\mu'-\nu'\mu\right)^2.
\end{eqnarray}
\end{lemma}
\begin{proof}
Recalling (\ref{param ellipse theta}) we have
\begin{equation}\label{matrices XY theta}
\left(\begin{array}{c}X\\Y\end{array}\right)=\left(\begin{array}{cc}\mu&\nu\\\mu'&\nu'\end{array}\right)\left(\begin{array}{c}\cos 2\theta\\\sin 2\theta\end{array}\right).
\end{equation}
Thus
\begin{equation}\label{equation theta mu nu}
(\mu\nu'-\mu'\nu)\left(\begin{array}{c}\cos 2\theta\\\sin 2\theta\end{array}\right)=\left(\begin{array}{cc}\nu'&-\nu\\-\mu'&\mu\end{array}\right)\left(\begin{array}{c}X\\Y\end{array}\right).
\end{equation}
The identity $\cos^22\theta+\sin^22\theta=1$ gives the result.
\end{proof}
\begin{lemma}\label{lemma equation ellipse}
If $K_N\neq 0,$ equation (\ref{equation ellipse mu nu}) is the equation of the curvature ellipse. If $K_N=0,$ the curvature ellipse is the segment defined by
\begin{equation}\label{equation segment}
\sqrt{\nu'^2+\mu'^2}X=\pm\sqrt{\nu^2+\mu^2}Y
\end{equation}
with $-\sqrt{\nu^2+\mu^2}\leq X\leq\sqrt{\nu^2+\mu^2}$ and $-\sqrt{\nu'^2+\mu'^2}\leq Y\leq\sqrt{\nu'^2+\mu'^2}.$ The sign in (\ref{equation segment}) is the sign of $\nu\nu'+\mu\mu'.$
\end{lemma}
\begin{proof}
We first assume that $K_N\neq 0,$ and we suppose that $(X,Y)$ is a solution of (\ref{equation ellipse mu nu}). Equation (\ref{matrices XY theta}) in $\theta$ is solvable if and only if
$$\frac{1}{\left(\nu\mu'-\nu'\mu\right)^2}\left((\nu'X-\nu Y)^2+(-\mu' X+\mu Y)^2\right)=1,$$
which is exactly (\ref{equation ellipse mu nu}). This gives the result.

We now assume that $K_N=0.$ The Lagrange identity thus gives
\begin{equation}\label{Lagrange identity K_N=0}
(\nu'\nu+\mu\mu')^2=(\nu^2+\mu^2)(\nu'^2+\mu'^2).
\end{equation}
Equation (\ref{equation ellipse mu nu}) thus reads
$$\left(\nu'^2+\mu'^2\right)X^2+\left(\nu^2+\mu^2\right)Y^2\pm 2\sqrt{\nu^2+\mu^2}\sqrt{\nu'^2+\mu'^2}XY=0,$$
where the sign $\pm$ is the opposite sign of $\nu'\nu+\mu\mu.$ Thus
$$\left(\sqrt{\nu'^2+\mu'^2}X\pm\sqrt{\nu^2+\mu^2}Y\right)^2=0,$$
which gives (\ref{equation segment}). The inequalities on $X,Y$ directly come from (\ref{matrices XY theta}).
\end{proof}
We now link the quadratic form $\Phi$ to the curvature ellipse. Let us denote by $\overline{E}$ the curvature ellipse with its interior points
$$\overline{E}=\left\{q(u):\ u\in\R^2,\ |u|\leq 1\right\}.$$
We first define as in the euclidean space  the support function $h$ of the convex body $\overline{E}\subset\R^{1,1}$ with respect to some interior point $O'\in \overline{E}$ (taken as the new origin of $\R^{1,1}$): for all $\nu\in\R^{1,1},$
$$h(\nu):=\sup_{\nu'\in \overline{E}}\langle\nu',\nu\rangle .$$
\begin{lemma}
The support function of the curvature ellipse with respect to its interior point $\H$ is the square root of $\Phi$: for all $\nu\in\R^{1,1},$ 
$$h(\nu)=\sqrt{\Phi(\nu)}.$$
In particular, if $K_N=0$ and if $\xi\in\R^{1,1}$ is such that the curvature ellipse is the segment $[\H-\xi,\H+\xi],$ then, for all $\nu\in\R^{1,1},$
\begin{equation}\label{Phi segment}
\Phi(\nu)=\langle\xi,\nu\rangle^2.
\end{equation}
\end{lemma}
\begin{proof}
By the very definition of $h,$ and recalling the notation used in the proof of Lemma \ref{identity Phi,A},
$$h(\nu)=\sup_{\{u\in\R^2:\ |u|\leq 1\}}\langle II(u)-\H,\nu\rangle=\sup_{\{u\in\R^2:\ |u|\leq 1\}}\langle S_{\nu}^o(u),u\rangle=\|S_{\nu}^o\|_F.$$ 
By (\ref{interpretation Phi}) this last quantity is $\sqrt{\Phi(\nu)}.$ We then obtain the last claim of the lemma from the definition of $h.$
\end{proof}
If $K_N\neq 0$ the function $u_{\Phi}$ is invertible, and we define
\begin{equation}\label{Phi star}
\Phi^*(\nu):=\langle\nu,u_{\Phi}^{-1}(\nu)\rangle.
\end{equation}
The function $\Phi^*:\R^{1,1}\rightarrow \R$ furnishes an intrinsic equation of the curvature ellipse:
\begin{lemma}\label{lemma ellipse Phi star} We suppose that $K_N\neq 0.$ For all $\nu\in\R^{1,1},$ $\H+\nu$ belongs to the ellipse $E$ if and only if
$$\Phi^*(\nu)=1.$$
\end{lemma}
\begin{proof}
Using (\ref{matrix u_Phi}) we get
$$\Phi^*(\nu)=\frac{1}{\left(\nu\mu'-\nu'\mu\right)^2}\left(\left(\nu'^2+\mu'^2\right)X^2+\left(\nu^2+\mu^2\right)Y^2-2\left(\nu'\nu+\mu\mu'\right)XY\right),$$
which implies the result since the equation of the curvature ellipse is given by (\ref{equation ellipse mu nu}).
\end{proof}
We deduce the following descriptions of the ellipse:
\begin{proposition}\label{lemma no degenerated ellipse}
If $K_N\neq 0,$ the curvature ellipse is not degenerated; its axis are directed by the eigenvectors of $u_\Phi,$ and the squared length of the semi-axis are the eigenvalues $a^2,\ -b^2$ of $u_{\Phi}$ given by (\ref{equation a}),(\ref{equation b}).
\end{proposition}
\begin{proof}
Lemma \ref{lemma ellipse Phi star} and expression (\ref{Phi star}) imply that the curvature ellipse is the set
$$E=\{\H+\nu\in\R^{1,1}:\ \langle u_{\Phi}^{-1}(\nu),\nu\rangle=1\}.$$
Writing $\nu=\nu_1\tilde u_1+\nu_2\tilde u_2$ in the basis $(\tilde u_1,\tilde u_2)$ of eigenvectors of $u_{\Phi}$ given by Lemma \ref{case Phi diag}, we obtain the following: in $(\H,\tilde u_1,\tilde u_2),$ the equation of the ellipse is given by
\begin{equation}\label{Phi star diag}
\frac{1}{a^2}\nu_1^2+\frac{1}{b^2}\nu_2^2=1.
\end{equation}
This implies the proposition.
\end{proof}
\begin{remark}\label{sign Phi star}
The expression of the quadratic form $\Phi^*$ given by the left-hand side term of (\ref{Phi star diag}) shows that it is positive definite.
\end{remark}
\begin{remark}
The sign of $K_N$ has the following meaning: $K_N$ is positive if the ellipse is described by (\ref{param ellipse theta}) in the positive direction when the parameter $\theta$ grows; it is negative otherwise. 
\end{remark}
\begin{proposition}\label{lemma degenerated ellipse}
If $K_N=0$ and $\Phi\neq 0,$ the curvature ellipse degenerates to a segment $[\H-\xi,\H+\xi],$ which is:
\\1- spacelike if $|\H|^2-K>0;$ more precisely, $\xi=\pm a\tilde u_1$ where $a=\sqrt{|\H|^2-K}$ and $\tilde u_1$ is the spacelike unit eigenvector of $u_{\Phi}$ given by Lemma \ref{case Phi diag};
\\2- timelike if $|\H|^2-K<0;$ more precisely, $\xi=\pm b\tilde u_2$ where $b=\sqrt{K-|\H|^2}$ and $\tilde u_2$ is the timelike unit eigenvector of $u_{\Phi}$ given by Lemma \ref{case Phi diag}; 
\\3- lightlike if $|\H|^2-K=0;$ recalling the notation in Lemma \ref{case Phi no diag}, $\xi=\pm\tilde N_2$ if the matrix of $u_{\Phi}$ in $(\tilde N_1,\tilde N_2)$ is $A_1$, and $\xi=\pm\tilde N_1$ if the matrix is $A_2.$
\end{proposition}
\begin{proof}
From Lemma \ref{lemma equation ellipse}, the curvature ellipse is the segment $[\H-\xi,\H+\xi]$ with
$$\xi=\sqrt{\mu^2+\nu^2}N_1\pm\sqrt{\mu'^2+\nu'^2}N_2,$$
where the sign $\pm$ is the sign of $\nu\nu'+\mu\mu'.$  By a direct computation (recalling (\ref{matrix u_Phi})),
$$u_{\Phi}(\xi)=-\left((\nu'\nu+\mu\mu')\pm\sqrt{\nu^2+\mu^2}\sqrt{\mu'^2+\nu'^2}\right)\xi.$$
Using the formulas (\ref{formula 1 mu nu}) and (\ref{formula 2 mu nu}) with $K_N=0,$ we get
\begin{equation}\label{expr Phi 1}
|\xi|^2=|\H|^2-K\mbox{ and }u_{\Phi}(\xi)=(|\H|^2-K)\xi,
\end{equation}
and thus the proposition in the cases where $|\H|^2-K\neq 0.$

If now $|\H|^2-K=0,$ then $\xi$ is lightlike. If the matrix of $u_{\Phi}$ in $(\tilde N_1,\tilde N_2)$ is $A_1,$ we get using (\ref{Phi segment})
$$\Phi(\tilde N_1)=\langle\xi,\tilde N_1\rangle^2=1\mbox{ and }\Phi(\tilde N_2)=\langle\xi,\tilde N_2\rangle^2=0.$$
Thus $\xi=\pm\tilde N_2.$
\end{proof}
\begin{remark}
If $\Phi=0$ the ellipse degenerates to the point $\H.$
\end{remark}
We finally summarize the results of Propositions \ref{lemma no degenerated ellipse} and \ref{lemma degenerated ellipse}, and describe in each case the corresponding classes of quadratic maps (using Theorem \ref{theorem algebraic classification}). We set
\begin{eqnarray*}
\tilde{q}:\R^{1,1}&\rightarrow &\mbox{End}_{\mbox{Sym}}(\R^2)\\
\nu&\mapsto & S_{\nu}.
\end{eqnarray*}
\begin{theorem} \label{theorem ellipse classification}
Let $r:=rank(\tilde{q})$ be the rank of $q.$ We distinguish three main cases, according to the value of $r:$
\\
\\1- $r=0$ (\emph{umbilicity}): the curvature ellipse reduces to the point $O.$ In that case the class $[q]$ is $[0].$ 
\\
\\2- $r=1$ (\emph{inflection}): $K_N=0,\ \Delta=0.$ The curvature ellipse is a segment centered at $\H$ which belongs to the line $\R.\H$. The segment is spacelike, timelike or lightlike if $|\H|^2-K>0,$ $|\H|^2-K<0,$ or $|\H|^2-K=0.$ 
\\There are two cases:
\\a- $|\H|^2-K\neq 0.$ The class $[q]$ is then determined by the values of $|\H|^2$ and $K$ (up to the action of $G$); 
\\b- $|\H|^2-K= 0.$ Then $|\H|^2=0,\ K=0,$ and the class $[q]$ is determined by the value of $\zeta$ (up to the action of $G'$).
\\
\\3- $r=2.$ There are two cases:
\\a- $\ K_N\neq 0$ (\emph{regularity}) The curvature ellipse is an ellipse centered at $\H.$ The class $[q]$ is then determined by the values of $|\H|^2,\ K,\ K_N$ and $\Delta$ if $\H\neq\vec{0}$ (up to $\pm id_{\R^{1,1}}$); if $\H=\vec{0},$ $[q]$ is determined by the values of  $K$ and $K_N$ (in that case $\Delta=-\frac{1}{4}K_N^2).$
\\b- $\H\neq \vec{0},\ K_N=0, \Delta\neq 0$ (\emph{semi-umbilicity}). The curvature ellipse is a segment centered at $\H,$ which does not belong to the line $\R\H.$ The class $[q]$ is determined by the values of $|\H|^2,\ K$ and $\Delta$ (up to $G$).
\end{theorem}
We will need the following lemma for the proof:
\begin{lemma} \label{lemma prop segment}We suppose that the curvature ellipse degenerates to the segment $[\H-\xi,\H+\xi],$ with $\xi\in\R^{1,1}.$ Then
\begin{equation}\label{equation Delta H}
\Delta=\left[\H,\xi\right]^2.
\end{equation}
Moreover $\Delta=0$ if and only if $\mbox{rank}\ \tilde q=1.$ This equivalently means that the ellipse belongs to the line $\R.\H.$ 
\end{lemma}
\begin{proof}
According to (\ref{expr Phi 1}), $|\xi|^2= |\H|^2-K.$ Formula (\ref{Delta function xi}) thus implies (\ref{equation Delta H}). 

Since $K_N=0,$ the morphisms $S_{\nu},\ \nu\in\R^{1,1}$ commute; thus there exists a basis of $\R^2$ in which their matrices are diagonal with diagonal entries $\lambda_1(\nu),\lambda_2(\nu),$ with $\lambda_1,\lambda_2\in{\R^{1,1}}^*.$ We have $Q(\nu)=\lambda_1(\nu)\lambda_2(\nu),$ and we thus see that $Q$ is not degenerated if and only if $\lambda_1$and $\lambda_2$ are independent. Thus $\Delta=0$ if and only $\mbox{rank}\ \tilde q=1.$ The last claim follows from (\ref{equation Delta H}).
\end{proof}
\textit{Proof of the Theorem \ref{theorem ellipse classification}.}
If $r=0$ then $q=0$ and the result is obvious.

If $r=1,$  the morphisms $S_{\nu},\ \nu\in\R^{1,1}$ commute; this implies that $K_N=0.$ The results then follow from Proposition \ref{lemma degenerated ellipse}, Lemma \ref{lemma prop segment} and Theorem \ref{theorem algebraic classification}.

We now suppose that $r=2.$ If $K_N\neq 0$ the results follow from Lemma \ref{lemma no degenerated ellipse} and Theorem \ref{theorem algebraic classification} (for the relation between $\Delta$ and $K_N$ when $\H=\vec{0},$ see (\ref{formula Delta a}) and (\ref{formula Delta b})). If $K_N=0,$  then $\Delta\neq 0,$ $\H\neq 0$ and the segment does not belong to the line $\R.\H$ (by Lemma \ref{lemma prop segment}); we then conclude with Theorem \ref{theorem algebraic classification}.\begin{flushright}${\Box}$\end{flushright}

\section{Principal configurations on spacelike surfaces}\label{section principal configurations}
Let us apply results proved in the previous sections to spacelike surfaces immersed in $\mathbb R^{3,1}$. We suppose that $\R^{3,1}$ is oriented by its canonical basis, and that it is \textit{time-oriented} as follows: we will say that a vector of $\R^{3,1}$ is future-directed if its last component in the canonical basis is positive. Let us consider a smooth oriented spacelike surface $M$ immersed in $\mathbb R^{3,1}.$ Note that each normal plane of $M$ is then naturally oriented and time-oriented. See \cite{Oneill} as a general reference.

The second fundamental form at each point $p$ of $M$ is a quadratic map
$$II_p:T_pM \rightarrow N_pM.$$
Observe that if we choose orthonormal basis of $T_pM$ and of $N_pM$ which are positively oriented (the second vector of the basis of $N_pM$ being moreover timelike and future-directed), $T_pM$ identifies with $\R^2$ and $N_pM$ with $\R^{1,1}$ (with their canonical orientations). Therefore, the forms, the invariants and the ellipse of curvature studied in the previous sections appear in this setting.

For $p\in M$ and $\nu \in N_pM,\ \nu\neq 0,$ the {\it second fundamental form with respect to $\nu$} at $p$ is the quadratic form  $II_{\nu}:T_pM \rightarrow \R$ defined by 
$$II_{\nu}(X)= \langle II_p(X,X),\nu\rangle. $$
Let $\bar{\nu}$ be a local extension to $\mathbb R^{3,1}$ of the normal vector $\nu $ of $M$ at $p$. The {\it shape operator} $S_{\nu }:T_pM\rightarrow T_pM$ defined by 
$$S_{\nu }(X)=-(\bar{D}_{\bar{X}}\bar{\nu})^{\top }, $$
is the self-adjoint operator associated to the quadratic form $II_{\nu}:$ for all $X,Y\in T_pM,$
$$\left\langle S_{\nu }(X),Y\right\rangle =\left\langle II_p\left(X,Y\right),\nu \right\rangle.$$
We can find, for each $p \in M,$ an orthonormal basis of eigenvectors of $S_{\nu}$ in $T_pM$, for which the restriction of the second fundamental form to the unitary tangent vectors $II_{\nu}|_{\mathbb S^1}$, takes its maximal and minimal values. The corresponding eigenvalues $\lambda_1$, $\lambda_2$ are the {\it $\nu$-principal curvatures}. A point $p$ is said to be $\nu$-umbilic if both $\nu$-principal curvatures coincide at $p.$ 

Let ${\bf U}_{\nu}$ be the set of $\nu$-umbilics in $M$. For any $p \in M \setminus {\bf U}_{\nu}$, there are two $\nu$-principal directions defined by the eigenvectors of $S_{\nu}$. These fields of directions are smooth and integrable, and they define two families of orthogonal curves, their integrals, which are the $\nu$-principal lines of curvature. The two orthogonal foliations with the $\nu$-umbilics as their singularities form the $\nu$-principal configuration of $M$.   

The differential equation of the $\nu$-principal lines of curvature is given by
\begin{eqnarray} \label{equation nu lines}
S_{\nu}(X(p))= \lambda(p) X(p),
\end{eqnarray}
where $\lambda(p) \in \R$. Now, let us obtain the expression of this differential equation in a coordinate chart. Let $\phi $ be a parameterization of an open neighborhood $U\subset M$ with local coordinates $(u,v)$, and let $\nu$ be a normal vector field along $U$. For each $p=\phi(u,v)\in M$, the associated basis of $T_{p}M$ is $(\phi _{u}=\frac{\partial\phi}{\partial u},\phi _{v}=\frac{\partial\phi }{\partial v}).$ The coefficients of the second fundamental form with respect to $\nu$ are 
\begin{eqnarray*}
e_{\nu} & = & II_{\nu}\left( \phi_u \right) = <II(\phi_u,\phi_u), \nu>,\ f_{\nu} =<II(\phi_u,\phi_v),\nu >,\\
g_{\nu} & = & II_{\nu}\left( \phi_v \right) = <II(\phi_v,\phi_v), \nu> . 
\end{eqnarray*}

A standard procedure for elimination of the parameter $\lambda$ in (\ref{equation nu lines}) provides the expression of the equation of the $\nu$-principal lines of curvature in this coordinate chart:
\begin{equation}\label{difeq}
(Fe_{\nu}-Ef_{\nu})du^2+(Ge_{\nu}-Eg_{\nu})dudv+(Gf_{\nu}-Fg_{\nu})dv^2=0,
\end{equation}
\noindent where $E, F$ and $G$ are the coefficients of the first fundamental form in this coordinate chart. 

The study of the $\nu$-principal configurations on surfaces immersed in $\mathbb R^4$ from the dynamical view point has been developed in \cite{Ra-SB} and \cite{Ro-SB}. In the present work the causal character of the normal vector field  will play an important role. A spacelike surface immersed in $\mathbb R^{3,1}$ has a well defined lightcone Gauss map whose associated shape operator is known as the normalized shape operator, and the corresponding principal configuration is known as the lightcone principal configuration; see \cite{Izumiya-Romero2007} and \cite{IzumiyaNuñoRomero2004}. 

Observe that the lightcone principal configuration is a particular case of $\nu$-principal configuration. Indeed, the lightcone Gauss map is defined as follows. Any timelike normal vector field $n^t$ determines, by means of the cross product of this vector field with the standard tangent frame of the immersion, a unique oriented spacelike normal vector field $n^s$. It turns out that $n^t + n^s$ is a lightlike normal vector field. Moreover, the direction determined by this vector field is independent of the initial vector field $n^t$. Namely, if  $\{ \bar n^t, \bar n^s\}$ is another normal frame determined  by $\bar n^t$ as above, the vector fields $\bar n^t + \bar n^s$ and $n^t + n^s$ are parallel (Lemma 3.1, \cite{Izumiya-Romero2007}). The  lightcone principal configuration is defined to be the $(n^t + n^s)$-principal configuration. 

In the case of a surface immersed in the hyperbolic space, $\phi:M \rightarrow H_{+}^3(-1),$ for the timelike normal vector field at $\phi(p)$ we may choose the position vector $\vect{o\phi(p)},$ where $o$ is the origin of $\R^{3,1};$ the vector $n^s$ is then univocally defined on the de Sitter space $S^3_1$.  The lightcone Gauss map coincides with the hyperbolic Gauss map and the corresponding shape operator is the horospherical shape operator whose principal directions define the horospherical principal configuration \cite{Izumiya2003}. This configuration coincides with the principal configuration of the surface immersed in $H_+^3(-1),$ since such a surface is semi-umbilic in $\R^{3,1}.$

\subsection{$\nu$-principal lines at the neighborhood of a lightlike umbilical point}\label{normal form}
Let us analyze the $\nu$-principal configuration near an isolated $\nu$-umbilic. If the vector field $\nu$ is spacelike or timelike the treatment is analogous to that of an isolated $\nu$-umbilic of a surface immersed in $\mathbb R^4$ \cite{Ra-SB}. Therefore, we consider the case of an isolated $\nu$-umbilic when the normal vector $\nu(p)$ is a lightlike vector. 

We fix a basis $(e_1,e_2,N_1,N_2)$ of $\R^{3,1}$ adapted to the splitting $\R^{3,1}=T_pM\oplus N_pM$ at the $\nu$-umbilic $p:$ $(e_1,e_2)$ is a positively oriented and orthonormal basis of $T_pM$ and $(N_1,N_2)$ is a positively oriented basis of null and future-directed vectors of $N_pM$ such that $\langle N_1,N_2\rangle=-1.$ In the basis $(e_1,e_2,N_1,N_2),$ we suppose that the immersion $\phi$ is of the form
$$\phi(u,v)=(u,v,A(u,v),B(u,v))$$
with
$$A(u,v)=\frac{1}{2}(a_{20}u^2+2a_{11}uv+a_{02}v^2)+\frac{1}{6}(a_{30}u^3+3a_{21}u^2v+3a_{12}uv^2+a_{03}v^3)+o(3)
$$
and
$$B(u,v)=\frac{k}{2}(u^2+v^2)+\frac{1}{6}(b_{30}u^3+3b_{21}u^2v+3b_{12}uv^2+b_{03}v^3)+o(3).$$
This means that, at $p=(0,0),$
$$II=\left(\begin{array}{cc}a_{20}&a_{11}\\a_{11}&a_{02}\end{array}\right)N_1+\left(\begin{array}{cc}k&0\\0&k\end{array}\right)N_2,$$
or equivalently that $II_{N_1}=-k\langle.,.\rangle.$ Thus, $(0,0)$ is a $N_1-$umbilic of the surface. We have 
\begin{eqnarray*}
\phi_u & = & (1,0,a_{20}u+a_{11}v+\frac{1}{2}(a_{30}u^2+2a_{21}uv+a_{12}v^2)+o(2),\\
&  & ku + \frac{1}{2}(b_{30}u^2+2b_{21}uv+b_{12}v^2) + o(2)), \\ 
\phi_v & = &(0,1,a_{11}u+a_{02}v+\frac{1}{2}(a_{21}u^2+2a_{12}uv+a_{03}v^2)+ o(2),\\
& & kv+\frac{1}{2}(b_{21}u^2+2b_{12}uv+b_{03}v^2) + o(2)), \\ 
\phi_{uu} & = & (0,0,a_{20}+a_{30}u+a_{21}v+ o(1), 
k+b_{30}u+b_{21}v+ o(1)),\\
\phi_{uv} & = & (0,0,a_{11}+a_{21}u+a_{12}v+ o(1), 
b_{21}u+b_{12}v+ o(1))
\end{eqnarray*}
and
\begin{eqnarray*}
\phi_{vv}&=&(0,0,a_{02}+a_{12}u+a_{03}v+ o(1),
k+b_{12}u+b_{03}v+ o(1)).
\end{eqnarray*}
Let
$$\nu=\nu_1e_1+\nu_2e_2+\nu_3N_1+\nu_4N_2$$
be a normal field such that $\nu_1(0)=\nu_2(0)=\nu_4(0)=0$ and $\nu_3(0)=1.$ Define, for $i=1,2,3,4,$ the numbers $l_i,m_i,n_i$ 
such that
$$\nu_i=l_i+m_iu+n_iv+ o(1).$$
We readily have
$l_1=l_2=l_4=0$ and $l_3=1.$ Since
\begin{eqnarray*}
\langle \phi_u,\nu\rangle 
&=&(m_1-k)u+n_1v+ o(1)
\end{eqnarray*}
and
\begin{eqnarray*}
\langle \phi_v,\nu\rangle 
&=&m_2u+(n_2-k)v+ o(1),
\end{eqnarray*}
we get
$$m_1=k,\ n_1=0,\ m_2=0\mbox{ and }\ n_2=k.$$
By scaling $\nu,$ we may suppose without loss of generality that $\nu_3$ is a constant equal to 1. Thus $l_3=1$ and $m_3=n_3=0$. We denote $m_4,n_4$ by $m,n$ respectively. We thus get 
\begin{equation}\label{DLnu}
\nu_1=ku+ o(1),\ \nu_2=kv+ o(1),\ \nu_3=1,\mbox{ and }\nu_4=mu+nu+ o(1).
\end{equation}
Thus
\begin{eqnarray*}
e_{\nu}=\langle \phi_{uu},\nu\rangle&=&-(k+(b_{30}+ma_{20})u+(b_{21}+na_{20})v+ o(1)),
\end{eqnarray*}
\begin{eqnarray*}
f_{\nu}=\langle \phi_{uv},\nu\rangle&=&-((b_{21}+ma_{11})u+(b_{12}+na_{11})v+ o(1)),
\end{eqnarray*}
\begin{eqnarray*}
g_{\nu}=\langle \phi_{vv},\nu\rangle&=&-(k+(b_{12}+ma_{02})u+(b_{03}+na_{02})v+ o(1)).
\end{eqnarray*}
Since
$$E=\langle \phi_u,\phi_u\rangle=1 + o(1),\ F=\langle \phi_u,\phi_v\rangle= o(1)$$
and 
$$G=\langle \phi_v,\phi_v\rangle=1 + o(1),$$
we get
\begin{eqnarray*}
Fe_{\nu}-Ef_{\nu}&=&(b_{21}+ma_{11})u+(b_{12}+na_{11})v + o(1),
\end{eqnarray*}
\begin{eqnarray*}
Ge_{\nu}-Eg_{\nu}&=&(b_{12}-b_{30}+m(a_{02}-a_{20}))u+(b_{03}-b_{21}+n(a_{02}-a_{20}))v+o(1)
\end{eqnarray*}
and
\begin{eqnarray*}
Gf_{\nu}-Fg_{\nu}&=&-(b_{21}+ma_{11})u-(b_{12}+na_{11})v + o(1).
\end{eqnarray*}
We moreover suppose that the normal field is lightlike on a neighborhood of the point $p;$ this condition reads $\nu_1^2+\nu_2^2-2\nu_3\nu_4\equiv 0,$ which implies that $m=n=0.$ The expressions above then yield
\begin{eqnarray*}
Fe_{\nu}-Ef_{\nu}&=&b_{21}u+b_{12}v+ o(1),
\end{eqnarray*}
\begin{eqnarray*}
Ge_{\nu}-Eg_{\nu}&=&(b_{12}-b_{30})u+(b_{03}-b_{21})v+ o(1)
\end{eqnarray*}
and
\begin{eqnarray*}
Gf_{\nu}-Fg_{\nu}&=&-(b_{21}u+b_{12}v+ o(1)),
\end{eqnarray*}
and we thus obtain the equation of the lightcone principal curvature lines.

\begin{theorem}
For a generic spacelike embedding of a surface in $\R^{3,1},$ the differential equation of the lightcone principal curvature lines with an isolated umbilic has the following normal form:
\begin{eqnarray*}
&& (b_{21}u+b_{12}v+ o(1))du^2
+((b_{12}-b_{30})u+(b_{03}-b_{21})v+ o(1))dudv \\
&&-(b_{21}u+b_{12}v+ o(1))dv^2+o(1)=0.
\end{eqnarray*}
where $b_{ij}=\frac{\partial^{i+j} B}{\partial^iu \partial^jv}(0,0).$ 

The singular point is Darbouxian of type $D_1,\ D_2,\ D_3$  with index $\frac{1}{2}$ in case $D_1$ and $D_2$, and $-\frac{1}{2}$ in case $D_3$. 
\end{theorem}
We refer to \cite{Bruce} and \cite{Gut} for the description of Darbouxian types.

\subsection{Lightcone principal configurations on compact surfaces}\label{global results}

Let $M$ be a compact orientable surface and 
$$I=\{\phi:M \rightarrow \mathbb R^{3,1}|\ \phi \ {\rm is\ a\ smooth\ spacelike\ immersion}\}.$$ 
Endow this space with its standard topology. Consider the lightcone principal configuration of $\phi(M)$ described above.

\begin{theorem}\label{theorem Darbouxian}
The set of immersions $\phi \in I$ whose lightcone umbilic points are Darbouxian is an open and dense subset of $I$. 
\end{theorem}
Thus, from this theorem we get, as a consequence of the Poincar\'e-Hopf Theorem applied to the lightcone principal configurations the following result \cite{Fel}, \cite{IzumiyaNuñoRomero2004} :

\begin{theorem}\label{number lightlike umbilic points}
The number of lightlike umbilic points of any closed (compact without boundary) spacelike surface generically immersed in $\mathbb R^{3,1}$ is greater than or equal to $2|\chi(M)|$, where $\chi(M)$ denotes the Euler number of $M$. Consequently any spacelike $2$-sphere generically immersed in $\mathbb R^{3,1}$has at least $4$ lightlike umbilic points.
\end{theorem}

\begin{remark}
In \cite{IzumiyaNuñoRomero2004} (Theorem $7.4$ and Corollary $7.5$) the authors obtain results analogous to Theorems \ref{theorem Darbouxian} and \ref{number lightlike umbilic points} for the particular class of 
Horospherical configurations on surfaces immersed in $\mathbb H^3_+(-1)$. Nevertheless, they apply a different approach, that is, since any of these surfaces has a global parallel normal field, the position vector field of the immersion into $\mathbb H^3_+(-1)$, the horospherical principal configuration is diffeomorphically equivalent to the principal configuration of the image of some immersion of that surface into $\mathbb R^3$. Then they can apply those theorems that hold for surfaces immersed in $\mathbb R^3$ to the case under analysis. On the other hand, the local analysis on the $\nu$-principal configuration developed in the present article, restricted to the class of surfaces immersed in $\mathbb H^3_+(-1)$, provides different proofs of these results. 
 \end{remark}
We also have the following result analogous to Theorem $1.13$ in \cite{Little}.
\begin{theorem}\label{existence lightlike umbilic points}
Let $M$ be a compact orientable spacelike surface immersed in $\mathbb R^{3,1}$, with non-zero Euler characteristic. Then, $M$ has lightlike umbilic points. 
\end{theorem}

\section{Mean directionally curved lines and Asymptotic lines on spacelike surfaces in $\mathbb R^{3,1}$ and surfaces in $\mathbb R^4$}\label{section mean curvature lines}

If $p$ is a point on a surface $M$ immersed in $\mathbb R^4$, a mean directionally curvature direction in the tangent plane $T_pM$ is defined as the inverse image by the second fundamental form of the points in the ellipse of curvature where the line defined by the mean curvature vector intersects the ellipse. 

The field of directions, maybe with singularities, defined in this way is called the field of {\it mean 
directionally curvature directions}, and has been studied recently in \cite{Mello}. This definition naturally extends to the case of spacelike surfaces immersed in  $\mathbb R^{3,1}$. 

In the sequel the invariants and functions of Section \ref{section def invariants} that will be applied are considered evaluated at the point $p$ without any reference.   

Consider $II(u)=\vec{H}+II^{\circ}(u),$ where $u$ is a unit vector in $T_pM$ and $II^{\circ}$ is the traceless part of $II$. The condition that determines the mean curvature directions is
\begin{equation}\label{condition mean curvature directions} 
[\vec{H}, II^{\circ}(u)]=0,
\end{equation} 
where the brackets stand for the mixed product in $N_pM$ (the determinant in a positively oriented Lorentzian basis). We first express $II^\circ(u)$ in a basis adapted to the axis of the curvature ellipse. We use notation and results of Lemma  \ref{case Phi diag}. Choose first a unit vector $e_1$ of $T_{p}M$ such that $II^{\circ}(e_1)=a\tilde u_1.$

Since the map $II^{\circ}$ increases the angle two times, there is a vector $e_2$ orthogonal to $e_1$ satisfying the equation $II^{\circ}(e_2)= - a\tilde u_1$. Therefore,  
$$\frac{1}{2}II^{\circ}(e_1 + e_2)=II^{\circ}\left(\frac{\sqrt 2}{2}(e_1 + e_2)\right)= b \tilde u_2,$$  
where the last equality holds because  the unit vector $\frac{\sqrt 2}{2}(e_1 + e_2)$ makes an angle of $\frac{\pi}{4}$ with $e_1$. Thus, setting $u=x_1e_1+x_2e_2,$ we get 
\begin{eqnarray}\label{expression II circ}
II^{\circ}(u) & = & (x_1,x_2)\left(\begin{array}{cc} 
a & 0 \\
0 & -a
\end{array}\right)\left(\begin{array}{c}x_1\\x_2\end{array}\right) \tilde u_1+ (x_1,x_2) \left(\begin{array}{cc} 
0 & b \\
b & 0
\end{array}\right)\left(\begin{array}{c}x_1\\x_2\end{array}\right) \tilde u_2 \nonumber \\
& = & a (x_1^2 - x_2^2) \tilde u_1 + 2b x_1 x_2 \tilde u_2.
\end{eqnarray} 

With the expression above and recalling Lemma \ref{case Phi diag}, a straigthforward computation shows that equation (\ref{condition mean curvature directions}) is equivalent to
\begin{eqnarray}\label{equation mean curvature directions}
m(u):=-\beta a x_1^2 + 2 \alpha b x_1 x_2 + \beta a x_2^2=0.  
\end{eqnarray}
where $\alpha$ and $\beta$ are defined in Lemma \ref{case Phi diag}.

Let us consider now the field of { \it asymptotic directions} of a surface immersed in $\R^{3,1}$. Let $p \in M$; an asymptotic direction at $T_pM$ is defined as the inverse image of the second fundamental form of a point where the line that contains the origin is tangent to the ellipse of curvature; these line fields have been studied for surfaces in $\R^4$ in \cite{Little}, \cite{M-RF-R}, \cite{R-M-RF-R}, \cite{B-T} and \cite{Tari2008}.

We suppose that $K_N\neq 0,$ and we set, for any $\nu,\mu\in N_pM,$
\begin{equation}\label{def euclidean structure}
<\nu,\mu >_*:=<\nu,u_{\Phi}^{-1}(\mu)>.
\end{equation}
This defines a metric on $N_pM$ (see Remark \ref{sign Phi star}). We suppose in the following that $N_pM$ is equipped with this euclidean structure. By Lemma \ref{lemma ellipse Phi star}, the curvature ellipse $E$ is a unit circle of this euclidean plane (with center $\H$). Thus, the lines in $N_pM$ through the origin and tangent to $E$ are described by the equation
$$0= <\vec{H}+ II^{\circ}(u), II^{\circ}(u)>_*,$$
or equivalently by
\begin{eqnarray}\label{eqasym} 
<\vec{H},II^{\circ}(u)>_* =- 1.
\end{eqnarray}
Since $u_{\Phi}$ is diagonal in the basis $(\tilde u_1,\tilde u_2),$ with diagonal entries $(a^2,-b^2),$ we readily get that
$$<\nu,\mu >_*= \frac{\nu_1\mu_1}{a^2}+ \frac{\nu_2\mu_2}{b^2},$$ 
where $\nu = \nu_1 \tilde u_1  + \nu_2 \tilde u_2$ and $\mu = \mu_1 \tilde u_1  + \mu_2 \tilde u_2$ are normal vectors. Recalling that $\vec{H} = \alpha \tilde u_1 + \beta \tilde u_2$ and using (\ref{expression II circ}), equation (\ref{eqasym}) reads
\begin{eqnarray}\label{def quadratic form d}
\delta(u):=b(a +  \alpha )x_1^2+ 2a \beta x_1x_2+b(a-\alpha)x_2^2=0.
\end{eqnarray}
We also used here that $x_1^2+x_2^2=1.$ It is not difficult to verify that (\ref{def quadratic form d}) also holds in the case where $K_N=0.$ 

Let us denote by $u_{\delta}$ the symmetric operator of $T_pM$ such that, for all $u\in T_pM,$
$$\delta(u)=\langle u_{\delta}(u),u\rangle.$$
We now determine the eigenspaces of $u_\delta$ using the following elementary property: $u=x_1e_1+x_2e_2$ is an eigenvector of $u_\delta$ if and only if $\tilde \delta(u,u^{\perp})=0,$ where $\tilde \delta$ denotes the polar form of $\delta$ and $u^{\perp}$ stands for the vector $-x_2e_1+x_1e_2.$ This condition reads
$$(x_1,x_2)\left(\begin{array}{cc}b(a+\alpha)& a\beta\\a\beta&b(a-\alpha)\end{array}\right)\left(\begin{array}{r}-x_2\\x_1\end{array}\right)=0,$$
which is equation (\ref{equation mean curvature directions}) of the mean curvature directions. Thus, the mean curvature directions are the proper directions of the quadratic form $\delta.$ Since the proper directions of a quadratic form clearly bisect its null directions, we conclude:

\begin{proposition}\label{proposition bisect}
The mean curvature directions bisect the angle determined by the asymptotic directions. Moreover, these directions coincide with the proper directions of the quadratic form $\delta$ defined by the asymptotic directions.  
\end{proposition}
\begin{remark}
The quadratic form $\delta$ defined in (\ref{def quadratic form d}) is, up to sign, the lorentzian analog of the quadratic form $\delta$ introduced by J.Little in \cite{Little}. Its trace is 
$$\tr\delta:=\tr u_\delta=2ab=\pm K_N.$$ 
We compute its determinant 
$$\det\delta:=\det u_\delta=a^2b^2-b^2\alpha^2-a^2\beta^2.$$
Using (\ref{equation alpha}) and (\ref{equation beta}) in the form
$$\alpha^2=\frac{1}{a^2+b^2}\left(\Delta+a^2|\H|^2+\frac{1}{4}K_N^2\right)$$
and 
$$\beta^2=\frac{1}{a^2+b^2}\left(\Delta-b^2|\H|^2+\frac{1}{4}K_N^2\right).$$
we get by a direct computation that
\begin{equation}\label{formula for Phi star H}
b^2\alpha^2+a^2\beta^2=\Delta+\frac{1}{4}K_N^2.
\end{equation}
Thus
$$\det\delta=-\Delta.$$ 
In an orthonormal basis of eigenvectors of $\delta,$ equation (\ref{def quadratic form d}) simplifies to
$$\delta_1{x_1'}^2+\delta_2{x_2'}^2=0,$$
where $\delta_1$ and $\delta_2$ are such that $\delta_1+\delta_2=K_N$ and $\delta_1\delta_2=-\Delta.$ By an elementary computation we get the Lorentzian analog of Wong's formula \cite{W}
$$\tan^2\theta=\frac{4\Delta}{K_N^2},$$
where $\theta$ is the angle formed by the asymptotic directions. 
\end{remark}
\begin{remark}
Using formula (\ref{formula for Phi star H}) we readily get
\begin{equation}\label{Phi star H}
\langle\H,\H\rangle_*=\frac{4\Delta}{K_N^2}+1.
\end{equation} 
Here we also used that $a^2b^2=\frac{1}{4}K_N^2$. Thus $K_N^2\geq -4\Delta,$ with equality if and only if $\vec H=0$ or $p$ is an inflection point ($K_N=\Delta=0$). Moreover, since the curvature ellipse is
$$E=\{\H+\nu\in\R^{1,1}:\ \langle\nu,\nu\rangle_*= 1\}$$
we deduce the following result, which is well-known in the euclidean setting: denoting by $O$ the origin of the normal plane $N_pM,$

\textit{if $K_N\neq 0,$ then $\Delta=0$ if and only if $O$ belongs to $E,$ $\Delta>0$ if and only if $O$ is outside $E,$ and $\Delta<0$ if and only if $O$ is inside $E.$}
\end{remark}
\begin{remark}
It can be proved that the quadratic form $\delta$ of $M\subset\R^{3,1}$ is proportional to the form $\delta$ of $M$ considered as a surface in the euclidean space $\R^4.$ Thus, the notions of asymptotic lines in the euclidean and Minkowski spaces coincide. In fact, the asymptotic directions on a surface may be described in terms of its Gauss map only: let $\Sigma$ be a surface immersed in the affine space $\R^4,$ and let $G$ be the Gauss map
\begin{eqnarray*}
G:\Sigma&\rightarrow& G(2,4)\\
x&\mapsto& T_x\Sigma,
\end{eqnarray*}
where $G(2,4)$ is the Grassmannian of the 2-planes in $\R^4.$ Consider the Pl\"ucker embedding of $G(2,4)$ in $\mathbb{P}^5.$ Its image, denoted by $Q,$ is a smooth quadric of $\mathbb{P}^5,$ known as the Klein quadric. A line in $\mathbb{P}^5$ is said to be a null line if it belongs to $Q.$ The asymptotic directions of $\Sigma$ are in fact the directions $u\in T_x\Sigma$ such that $dG_x(u)$ is tangent to some null line of $Q.$ The notion of asymptotic directions is thus independent 
of the lorentzian or euclidean structure of the ambient space where $\Sigma$ is immersed. 
We refer to \cite{WW}, Sections 1.3 and 1.4 for the Pl\"ucker embedding and the linear geometry of the Klein quadric. This conclusion is in agreement to the fact that
the notion of asymptotic directions at a point on a surface immersed in Euclidean 4-space can be stated in terms of the contact of the surface with certain normal hyperplanes or certain lines, where this contact is of higher order, see \cite{M-RF-R} or \cite{B-T}, respectively.
\end{remark}

F. Tari has analyzed some differential equations on surfaces immersed in $\R^n,\ n\geq 4$. He shows that given a surface immersed in $\R^4$, there is a natural way to determine a unique binary differential equation which defines a set of orthogonal lines on the surface. Moreover, he proves that this set of lines is, in fact, the $\nu$-principal configuration on the surface, defined by some normal vector field $\nu$ (Theorem 2.5, \cite{Tari2008}).  

Because of this he calls the configuration of that lines, the principal configuration of the surface.

He proves in his article that the directions of this line field bisect the angle determined by the asymptotic lines (Corollary 4.4). Then, our Proposition $\ref{proposition bisect}$ implies that the  principal configuration of the surface is, in fact, the configuration defined by the mean directionally curved lines. We now explicit a normal vector field $\nu$ such that the mean directionally curved lines are the $\nu$-principal lines. Observe first that this description is possible only at points which are not semi-umbilics (i.e. such that $K_N\neq 0$). Recall the definition (\ref{def euclidean structure}) of the euclidean structure $\langle.,.\rangle_*$ on the normal planes of the surface. The principal directions of the real quadratic form $\langle II,\vec H\rangle_*$ are the critical points of $S^1\rightarrow \R,\ u\mapsto \langle II(u),\vec H\rangle_*,$ and are thus the directions $u\in S^1$ such that the tangent of the curvature ellipse at $II(u)$ is orthogonal to $\vec H$ (w.r.t. to $\langle.,.\rangle_*).$ Since the curvature ellipse is a circle for this scalar product, these directions $u$ are such that $II(u)$ belongs to the diameter of the circle which is directed by $\vec H.$ These directions are precisely the directions of the mean directionally curved lines. Since $\langle II,\vec H\rangle_*=\langle II,u_{\Phi}^{-1}(\vec H)\rangle,$ setting 
$$\nu=u_{\Phi}^{-1}(\vec H),$$
the mean directionally curved lines are thus the $\nu$-principal lines (on the set where $K_N\neq 0$).

\end{document}